\documentclass{amsart}

\usepackage[hyperindex]{hyperref}

\usepackage[nosections]{macs-ind-irr}

\begin{document}

\date{20 October 2007; Revised 3 June 2008}

\author{Marius Ionescu}
\address{Department of Mathematics \\ Cornell University \\ Ithaca, NY
14853-4201}

\email{mionescu@math.cornell.edu}

\author{Dana Williams}
\address{Department of Mathematics \\ Dartmouth College \\ Hanover, NH
03755-3551}

\title{\boldmath Irreducible Representations of Groupoid \cs-algebras}
\email{dana.p.williams@Dartmouth.edu}

\begin{abstract}
  If $G$ is a second countable locally compact Hausdorff groupoid with
  Haar system, we show that every representation induced from an
  irreducible representation of a stability group is irreducible.
\end{abstract}


\maketitle

\section{Introduction}
\label{sec:introduction}

To understand the fine structure of a \cs-algebra, a good first step
is to describe the primitive ideals in a systematic way.  Therefore
producing a prescription for a robust family of irreducible
representations is very important.  In the case of transformation
group \cs-algebras, it is well known that representations induced from
irreducible representations of the stability groups are themselves
irreducible.  (Furthermore, in many cases these representations
exhaust the irreducible representations, or at least their kernels
exhaust the collection of primitive ideals, and a fairly complete
description of the primitive ideal space is possible.  For a more
extensive discussion, see \cite{wil:crossed}*{\S \S8.2--3}.)  In the
separable case, the irreducibility of representations induced from
irreducible representations of the stability groups is due to Mackey
\cite{mac:pnasus49}*{\S6} (see also Glimm's
\cite{gli:pjm62}*{pp.~900--901}).  The result for general
transformation group \cs-algebras was proved in
\cite{wil:tams81}*{Proposition~4.2} (see also
\cite{wil:crossed}*{Proposition~8.27}).  The corresponding result for
groupoid \cs-algebras has been proved in an \emph{ad hoc} manner in a
number of special cases (see Example~\ref{ex-stab-grps} for specific
references).  In this note, we want to prove the
result for general \emph{separable} groupoids.  In so doing, we take
the opportunity to formalize the theory of inducing representations
from a general closed subgroupoid.  Of course, induction is treated in
Renault's thesis \cite{ren:groupoid}*{Chap.~II \S2}.  However, at the
time \cite{ren:groupoid} was written, Renault did not yet have the
full power of his disintegration theorem
(\cite{ren:jot87}*{Proposition~4.2} or see
\cite{muhwil:nyjm08}*{Proposition~7.8}) available.  Nor was Rieffel's
theory of Morita equivalence fully developed.  So it seems appropriate
to give a modern treatment here using a contemporary version of
Rieffel's theory, and the disintegration theorem in the form of the
equivalence theorem from \cite{mrw:jot87}*{Theorem~2.8}.

In Section~\ref{sec:induc-repr}, we derive the general process for
inducing groupoid representations of a closed subgroupoid $H$ of $G$
to $G$
--- actually, we induce representations 
from $\cs(H)$ to $\cs(G)$.  In
Section~\ref{sec:main-theorem}, we specialize to the case where $H$ is
an isotropy group, $H=G(u)=G^{u}_{u}:=\set{x\in G:s(x)=r(x)}$, and
prove the main result.

Throughout, $G$ will be a second countable locally compact Hausdorff
groupoid.  Second countability, in the form of the separability of
$\cs(G)$, is necessary in Section~\ref{sec:induc-repr} in order to
invoke the disintegration theorem.  Although separability might be
unnecessary in the proof of the main theorem, we felt that, as much of
the deep theory of groupoid \cs-algebras uses the disintegration
result in one form or another, there was little to be gained which
would justify the additional work of adjusting the proof to handle the
general case.  In addition, we always assume that $G$ and $H$ have
(continuous) Haar systems.  We adopt the usual conventions that
representations of \cs-algebras are nondegenerate and that
homomorphisms between \cs-algebras are necessarily $*$-preserving.

\section{Inducing Representations}
\label{sec:induc-repr}

We assume that $G$ is a second countable locally compact groupoid with
Haar system $\set{\lambda^{u}}_{u\in\go}$.  Let $H$ be a closed
subgroupoid of $G$ with Haar system $\set{\alpha^{u}}_{u\in\ho}$.
Since $H$ is closed in $G$, we also have $\ho$ closed in $\go$.  Then
$\GsH:=s^{-1}(\ho)$ is a locally compact free and proper right
$H$-space and we can form the imprimitivity groupoid $H^{G}$ as
follows.  The space
\begin{equation*}
  \GsH*_{s}\GsH:=\set{(x,y)\in \GsH\times \GsH:s(x)=s(y)}
\end{equation*}
is a free and proper right $H$-space for the diagonal action
$(x,y)\cdot h:=(x h,y h)$.  Consequently, the orbit space
\begin{equation*}
  H^{G}:=(\GsH*_{s}\GsH)/H
\end{equation*}
is a locally compact Hausdorff space.  Following
\cite{mrw:jot87}*{\S2}, $H^{G}$ is a groupoid in a natural way.  If
$[x,y]$ denotes the orbit of $(x,y)$ in $H^{G}$, then the composable
pairs are given by
\begin{equation*}
  (H^{G})^{(2)}:=\set{\bigl([x,y],[z,w]\bigr):y\cdot H=z\cdot H},
\end{equation*}
and the groupoid operations are given by
\begin{equation*}
  [x,y][y h,z]:=[x,z
  h^{-1}]\quad\text{and}\quad[x,y]^{-1}:=[y,x]. 
\end{equation*}
We can identify $(H^{G})^{(0)}$ with $\GsH/H$ and then
\begin{equation*}
  r\bigl([x,y]\bigr)=x\cdot H\quad\text{and}\quad s\bigl([x,y]\bigr)
  =y\cdot H.
\end{equation*}
It is not hard to check that $H^{G}$ acts freely and properly on the
left of $\GsH$:
\begin{equation*}
  [x,y]\cdot (y h)=x h,
\end{equation*}
and that $\GsH$ is then a $(H^{G},H)$-equivalence as in
\cite{mrw:jot87}*{Definition~2.1}. 

To get a Haar system on $H^{G}$, we proceed as in
\cite{kmrw:ajm98}*{\S5}.  Since $\GsH$ is closed in $G$ and since
$\lambda$ is a Haar system on $G$, it is not hard to check that
\begin{equation*}
  \beta'(\phi)(x):=\int_{G} \phi(y)\,d\lambda_{s(x)}(y)\quad\text{$\phi \in
    C_{c}(\GsH)$} 
\end{equation*}
is a full equivariant $s$-system for the map $s:\GsH\to \ho$ (as
defined in \cite{kmrw:ajm98}*{\S5}).  Therefore
\cite{kmrw:ajm98}*{Proposition~5.2} implies that we get a Haar system
$\set{\beta^{x\cdot H}}_{x\cdot H\in (H^{G})^{(0)}}$ for $H^{G}$ via
\begin{equation*}
  \beta(F)(x\cdot H)=\int_{H^{G}}F\bigl([x,y]\bigr) \, d\beta^{x\cdot
    H}\bigl([x,y]\bigr) = \int_{G}F\bigl([x,y]\bigr)
  \,d\lambda_{s(x)}(y). 
\end{equation*}

Since both $H$ and $H^{G}$ have Haar systems,
\cite{mrw:jot87}*{Theorem~2.8} implies that $C_{c}(\GsH)$ is a
pre-$C_{c}(H^{G},\beta) \sme C_{c}(H,\alpha)$-\ib\ with actions and
inner products given by
\begin{align}
  \label{eq:1}
  F\cdot \phi(z)&=\int_{G} F\bigl([x,y]\bigr) \phi(y)\,d\lambda_{s(z)}
  (y) \\
\phi\cdot g(z)&=\int_{H} \phi(z h)g(h^{-1}) \,d\alpha^{s(z)} (h)
\label{eq:2}\\
\Rip<\phi,\psi>(h)&= \int_{G} \overline{\phi(y)} \psi(y h)\,
d\lambda_{r(h)} (y) \label{eq:3}\\
\Lip<\phi,\psi>\bigl([x,y]\bigr) &= \int_{H} \phi(y
h)\overline{\psi(x h)} \, d\alpha^{s(x)}(h).\label{eq:4}
\end{align}
We will write $\X=\X_{H}^{G}$ for the completion of $C_{c}(\GsH)$ as
a $\cs(H^{G})\sme\cs(H)$-\ib.

If $L$ is a representation of $\cs(H,\alpha)$, then we write $\xind L$
for the representation of $\cs(H^{G},\beta)$ induced via $\X$ (see the
discussion following \cite{rw:morita}*{Proposition~2.66}).  Recall
that $\xind L$ acts on the completion $\Hind$ of $C_{c}(\GsH)\atensor
\H_{L}$ with respect to the pre-inner product given on elementary
tensors by
\begin{equation*}
  \ip(\phi\tensor h|\psi\tensor
  k)=\bip(L\bigl(\Rip<\psi,\phi>\bigr)h|k). 
\end{equation*}
If $\phi\tensor_{H}h$ denotes the class of $\phi\tensor h$ in $\Hind$,
then
\begin{equation*}
  (\xind L)(F)(\phi\tensor_{H}h)=F\cdot \phi\tensor_{H}h.
\end{equation*}

To get an induced representation of $\cs(G)$ out of this machinery
(i.e., using \cite{rw:morita}*{Proposition~2.66}), we need a nondegenerate
homomorphism of \cs(G) into $\mathcal{L}(\X)$.  If $f\in C_{c}(G)$ and
$\phi\in C_{c}(\GsH)$ we can define
\begin{equation}\label{eq:10}
  f\cdot \phi(z)=\int_{G} f(y)\phi(y^{-1}z)\,d\lambda^{r(z)}(y).
\end{equation}

\begin{remark}
  \label{rem-convolution}
  Since $\GsH$ is closed in $G$, each $\phi\in C_{c}(\GsH)$ is the
  restriction of an element $f_{\phi}\in C_{c}(G)$.  Thus we can write
  \begin{equation*}
    \Rip<\phi,\psi>=\phi*\psi\quad\text{and}\quad f\cdot \phi=f*\phi,
  \end{equation*}
where, for example, $\phi*\psi$ should be interpreted as
$f_{\phi}*f_{\psi}$ restricted to $\GsH$ --- the point being that the
restriction is independent of our choice of $f_{\phi}$ and
$f_{\psi}$.  Similarly, $f*\phi$ is meant to be the restriction of
$f*f_{\phi}$ to $\GsH$.
\end{remark}

If $\rho$ is a state on $\cs(H)$, then
\begin{equation*}
  \ip(\cdot|\cdot)_{\rho}:=\rho\bigl(\Rip<\cdot,\cdot>\bigr)
\end{equation*}
is a pre-inner product on $C_{c}(\GsH)$ with Hilbert space completion
denoted by $\H_{\rho}$. If we let $V(f)\phi:=f\cdot\phi$, then
it follows from direct computation, or by invoking
Remark~\ref{rem-convolution} above, that
\begin{equation*}
  \Rip<V(f)\phi,\psi>=\Rip<f\cdot \phi,\psi> = \Rip<\phi,f^{*}\cdot
  \psi> = \Rip<\phi,V(f^{*})\psi>.
\end{equation*}
Thus $V$ induces a map of $C_{c}(G)$ into the
linear operators on the dense image of $C_{c}(\GsH)$ in $\H_{\rho}$
which clearly satisfies the
axioms of Renault's disintegration theorem (see, e.g.,
\cite{muhwil:nyjm08}*{Theorem~7.8} or
\cite{ren:jot87}*{Proposition~4.2}).  In particular, we obtain a
bona fide representation of $\cs(G)$ on $\Hind$,  and it follows that
\begin{equation*}
 \rho\bigl(\Rip<f\cdot \phi,f\cdot \phi>\bigr) \le \|f\|_{\cs(G)}^{2}
  \rho\bigl(\Rip<\phi,\phi>\bigr). 
\end{equation*}
Since this holds for all $\rho$,
\begin{equation*}
  \Rip<f\cdot \phi,f\cdot \phi>\le \|f\|^{2}\Rip<\phi,\phi>,
\end{equation*}
and $V(f)$ is a bounded adjointable operator on $\X$.  Therefore we
obtain an induced representation $\indgh L$ of $\cs(G)$ on $\Hind$
such that
\begin{equation*}
  (\indgh L)(f)(\phi\tensor_{H}h)=f* \phi\tensor_{H}h.
\end{equation*}

\begin{remark}
  \label{rem-composition}
  Since $M\bigl(\cs(H^{G})\bigr)\cong \mathcal{L}(\X)$
  (\cite{rw:morita}*{Corollary~2.54 and Proposition~3.8}), it is not
  hard to see that $\indgh L$ is the composition of $V$ with the
  natural extension of $\xind L$ to $\mathcal{L}(\X)$.
\end{remark}

\begin{example}
  \label{ex-stab-grps}
  In the next section, we will be exclusively interested in the
  special case of the above where $H$ is the stability group at a
  $u\in \go$.  That is,
  \begin{equation}
    \label{eq:5}
    H=G(u):=G^{u}_{u}=\set{x\in G:s(x)=u=r(x)}.
  \end{equation}
(Thus, $H^{(0)}=\set u$.)   In this case, we obtain the induced
representations used to establish special cases of
Theorem~\ref{thm-main} 
in  
\citelist{\cite{muhwil:ms90}*{Lemma~2.4}
\cite{muhwil:ms92}*{Lemma~3.2}
\cite{mrw:tams96}*{Lemma~2.5}
\cite{cla:iumj07}*{Lemma~4.2}\cite{facska:jot82}*{\S5}}.
\end{example}

One advantage of having a formal theory of induction for
representations of groupoid \cs-algebras is that we can apply the
Rieffel machinery.  An example is the following version of induction
in stages.  The proof, modulo technicalities, is a
straightforward modification of Rieffel's original ``\cs-version''
from \cite{rie:aim74}*{Theorem~5.9}.  For future reference, we've
worked out the details of the proof in the last section.

\begin{thm}[Induction in Stages]
  \label{thm-stages}
 Suppose that $H$ and $K$ are closed subgroupoids of a second
 countable locally compact Hausdorff groupoid $G$ with $H\subset K$.
 Assume that $H$, $K$ and $G$ have Haar systems. If $L$ is a
 representation of $\cs(H)$, then 
 \begin{equation*}
   \Ind_{H}^{G}L\quad\text{and}\quad \Ind_{K}^{G}\bigl(\Ind_{H}^{K} L\bigr)
 \end{equation*}
are equivalent representations of $\cs(G)$.
\end{thm}

\section{The Main Theorem}
\label{sec:main-theorem}

\begin{thm}
  \label{thm-main}
  Let $G$ be a second countable groupoid with Haar system
  $\set{\lambda^{u}}_{u\in\go}$.  Suppose that $L$ is an irreducible
  representation of the stability group $G(u)$ at $u\in\go$.  Then
  $\indgug L$ is an irreducible representation of $\cs(G)$.
\end{thm}

The idea of the proof is straightforward.  Let $L$ be an irreducible
representation of $\cs\bigl(G(u)\bigr)$.  Since $\X$ is a
$\cs\bigl(G(u)^{G}\bigr)\sme \cs\bigl(G(u)\bigr)$-\ib,
\cite{rw:morita}*{Corollary~3.32} implies that $\xind L$ is an
irreducible representation of $\cs\bigl(G(u)^{H}\bigr)$.  
We will show that any $T$ in the commutant of $\indgug L$ is a
scalar multiple of the identity.  It will suffice to see
that any such $T$ commutes with $(\xind L)(F)$ for all $F\in
C_{c}\bigl(G(u)^{G}\bigr)$.  Our proof will consist in
producing, given $F$, a net $\set{f_{i}}$ in $C_{c}(G)$ such that
\begin{equation*}
  (\indgug L)(f_{i})\to (\xind L)(F)
\end{equation*}
in the weak operator topology.  Since we will also arrange that this
net is uniformly bounded in the $\|\cdot\|_{I}$-norm on $C_{c}(G)$ ---
so that the net $\set{(\indgug L)(f_{i})}$ is uniformly bounded in
$B\bigl(\Hind)\bigr)$ --- we just have to arrange that
\begin{equation*}
  \bip((\indgug L)(f_{i})(\phi\tensor_{G(u)}h) |
  \psi\tensor_{G(u)}k) \to \bip((\xind L)(F) (\phi\tensor_{G(u)}h) |
  \psi\tensor_{G(u)}k)
\end{equation*}
for all $\phi,\psi\in C_{c}(G_{u})$ and $h,k\in\H_{L}$.

The next lemma is the essential ingredient to our proof.

\begin{lemma}
  \label{lem-key}
  Suppose that $F\in C_{c}\bigl(G(u)^{G}\bigr)$.  Then there is
  compact set $C_{F}$ in $G$ such that for each compact set $K\subset
  G_{u}$ there is a $f_{K}\in C_{c}(G)$ such that
  \begin{enumerate}
  \item $f_{K}(zy^{-1})=F\bigl([z,y]\bigr)$ for all $(z,y)\in K\times
    K$,
  \item $\supp f_{K} \subset C_{F}$ and
  \item $\|f_{K}\|_{I}\le \|F\|_{I}+1$.
  \end{enumerate}
\end{lemma}

The proof of Lemma~\ref{lem-key} is a bit technical, so we'll postpone
the proof for a bit, and show that the lemma allows us to prove
Theorem~\ref{thm-main}.

\begin{proof}[Proof of Theorem~\ref{thm-main}]
  For each $K\subset G_{u}$, let $f_{K}$ be as in
  Lemma~\ref{lem-key}.  Then $\set{f_{K}}$ and $\set{(\indgug
    L)(f_{K})}$ are nets indexed by increasing $K$. Notice that
  \begin{multline}\label{eq:9}
    \bip((\indgug L)(f_{K})(\phi\tensor_{G(u)}h) |
  \psi\tensor_{G(u)}k) -{}\\ \bip((\xind L)(F) (\phi\tensor_{G(u)}h) |
  \psi\tensor_{G(u)}k) \\
= \bip( L\bigl(\Rip<\psi,f_{K}*\phi-F\cdot \phi>\bigr)h|k)
  \end{multline}
Furthermore, using the invariance of the Haar system on $G$, we can
compute as follows:
\begin{equation}
  \label{eq:7}
  \begin{split}
    \Rip<\psi,f_{K}*\phi>(s)&= \int_{G}\overline{\psi(x)}
    f_{K}*\phi(xs) \, d\lambda_{u}(x) \\
&= \int_{G}\int_{G} \overline{\psi(x)} f_{K}(xz^{-1})\phi(zs)
\,d\lambda_{u}(z) \,d\lambda_{u}(x),
  \end{split}
\end{equation}
while on the other hand,
\begin{equation}
  \label{eq:8}
  \begin{split}
    \Rip<\psi,F\cdot \phi>(s)&= \int_{G}\overline{\psi(x)} F\cdot
    \phi(xs) \,d\lambda_{u}(x) \\
&=\int_{G}\int_{G} \overline{\psi(x)} F\bigl([xs,z]\bigr) \phi(z)
\,d\lambda_{u}(z) \,d\lambda_{u}(x) \\
&= \int_{G}\int_{G} \overline{\psi(x)} F\bigl([x,zs^{-1}]\bigr) \phi(z)
\,d\lambda_{u}(z) \,d\lambda_{u}(x) \\
&= \int_{G}\int_{G} \overline{\psi(x)} F\bigl([x,z]\bigr) \phi(z)
\,d\lambda_{u}(zs) \,d\lambda_{u}(x)
  \end{split}
\end{equation}

Notice that $\supp\Rip<\psi,\phi>\subset (\supp \psi)(\supp\phi)$.
Since $\supp f_{K}\subset C_{F}$ for all $K$, we have
\begin{equation*}
  \supp f_{K}*\phi\subset (\supp f_{K})(\supp \phi)\subset
C_{F}(\supp\phi).
\end{equation*}
  Therefore if \eqref{eq:7} does not vanish, then we
must have $s\in (\supp\psi)C_{F}(\supp \phi)$.  Therefore there is a
compact set $K_{0}$ --- which does \emph{not} depend on $K$ --- such
that both \eqref{eq:7} and \eqref{eq:8} vanish if $s\notin K_{0}$.
Thus if $s\in K_{0}$ and if $K \supset (\supp\psi)\cup (\supp
\phi)K_{0}^{-1}$, then the integrand in \eqref{eq:7} is either zero or
we must have $(x,z)\in K\times K$.  Therefore we can replace
$f_{K}(xz^{-1})$ by $F\bigl([x,z]\bigr)$ and $f_{K}*\phi-F\cdot\phi$
is the zero function  whenever $K$ contains $(\supp\psi)\cup (\supp
\phi)K_{0}^{-1}$.  Therefore the left-hand side of
\eqref{eq:9} is eventually zero, and the
theorem follows.
\end{proof}

We still need to prove Lemma~\ref{lem-key}, and to do that, we need some
preliminaries.  In the sequel, if $S$ is a Borel subset of $G$, then
\begin{equation*}
  \int_{S}f(x)\,d\lambda^{u}(x):=\int_{G}\charfcn{S}(x)f(x)\,d\lambda^{u}(x)
\end{equation*}
where $\charfcn{S}$ is the characteristic function of $S$.

\begin{lemma}
  \label{lem-bdd}
  Suppose that $f\in C_{c}^{+}(G)$ and that $K\subset G$ is a compact
  set such that
  \begin{equation*}
    \int_{K}f(x)\,d\lambda^{u}(x)\le M\quad\text{for all $u\in\go$.}
  \end{equation*}
There there is a neighborhood $V$ of $K$ such that
\begin{equation*}
  \int_{V}f(x)\,d\lambda^{u}(x)\le M+1\quad\text{for all $u\in\go$.}
\end{equation*}
\end{lemma}
\begin{proof}
  Let $K_{1}$ be a compact neighborhood
  of $K$.  Since $G$ is second countable, we can find a countable
  fundamental system $\set{V_{n}}$ of neighborhoods of $K$ in $K_{1}$;
  thus, given any neighborhood $V$ of $K$, there is a $n$ such that
  $V_{n}\subset V$ and 
$K=\bigcap V_{n}$.  Certainly, we can assume that $V_{n+1}\subset
V_{n}$.

If no $V$ as prescribed in the lemma exists, then for each $n$
we can find $u_{n}\in\go$
such that
  \begin{equation*}
    \int_{V_{n}}f(x)\,d\lambda^{u_{n}}(x)\ge M+1.
  \end{equation*}
Since we must have each $u_{n}\in r(K_{1})$, we can pass to a
subsequence, relabel, and assume that $u_{n}\to u_{0}$.  Since
$\charfcn{V_{n}}\to \charfcn{K}$ pointwise, the dominated convergence
theorem implies that
\begin{equation*}
  \int_{V_{n}}f(x)\,d\lambda^{u_{0}}(x)\to\int_{K}f(x)\,d\lambda^{u_{0}}(x).
\end{equation*}
In particular, there is a $n_{1}$ such that 
\begin{equation*}
  \int_{V_{n_{1}}}f(x)\,d\lambda^{u_{0}}(x)\le M+\frac12
\end{equation*}
Let $W$ be an open set such that $K\subset W\subset \overline{W}
\subset V_{n_{1}}$, and let $f_{0}\in C_{c}^{+}(G)$ be such that
$f_{0}\restr{\overline{W}}=f$, $f_{0}\le f$ and $\supp f_{0}\subset
V_{n_{1}}$.  Then
  \begin{equation*}
    \int_{G} f_{0}(x) \,d\lambda^{u_{0}}(x)\le M+\frac 12.
  \end{equation*}
However, since $\set{\lambda^{u}}$ is a Haar system,
\begin{equation*}
  \int_{G}f_{0}(x)\,d\lambda^{u_{n}}(x)\to
  \int_{G}f_{0}(x)\,d\lambda^{u_{0}}(x)\le M+\frac12.
\end{equation*}
But for large $n$, we have $V_{n}\subset W$ and therefore
\begin{equation*}
  \int_{G} f_{0}(x) \,d\lambda^{u_{n}}(x)\ge \int_{V_{n}}f_{0}(x)
  \,d\lambda^{u_{n}}(x) = \int_{V_{n}}f(x) \,d\lambda^{u_{n}}(x) \ge M+1.
\end{equation*}
This leads to a contradiction and completes the proof of the lemma.
\end{proof}

\begin{proof}[Proof of Lemma~\ref{lem-key}]
  The map $(z,y)\mapsto zy^{-1}$ is certainly continuous on
  $G_{u}\times G_{u}$ and factors through the orbit map
  $\pi:G_{u}\times G_{u}\to G(u)^{G}$.  In fact, if $zy^{-1}=xw^{-1}$,
  then we must have $z=x(w^{-1}y)$ and $y=w(x^{-1}z)$.  But
  $w^{-1}y=x^{-1}z$ and lies in $G(u)$.  Therefore, we have a
  well-defined injection $\Pi:G(u)^{G}\to G$ sending $[z,y]$ to
  $zy^{-1}$.  
We let $C_{F}$ be a compact neighborhood of $\Pi(\supp F)$.

Fix a compact set $K\subset G_{u}$.
The restriction of $\Pi$ to the compact set $\pi(K\times
  K)$ is a homeomorphism so we can find a function $\tfk\in C_{c}(G)$
  such that $\supp\tfk\subset C_{F}$ and such that $\tfk(zy^{-1}) =
  F\bigl([z,y]\bigr)$ for all $(z,y)\in K\times K$.

Let $K_{G}:=\pi(K\times K)$.  If
\begin{equation*}
  \int_{K_{G}}|\tfk(y)|\,d\lambda^{w}(y)\not=0,
\end{equation*}
then $K_{G}\cap G^{w}\not=\emptyset$.  Thus there is a $z\in K$ such
that $r(z)=w$ (and $s(z)=u$).  Then by left invariance
\begin{align*}
  \int_{K_{G}}|\tfk(y)|\,d\lambda^{w}(y)&=
  \int_{G}\charfcn{K_{G}}(zy)|\tfk(zy)|\, d\lambda^{u}(y) \\
&= \int_{G}\charfcn{K_{G}}(zy^{-1})|\tfk(zy^{-1})|\, d\lambda_{u}(y)\\
&= \int_{G} \charfcn{K_{G}}(zy^{-1}) |F\bigl([z,y]\bigr)| \,d\lambda_{u}(y) \\
&\le \|F\|_{I}.
\end{align*}
Similarly, if
\begin{equation*}
  \int_{K_{G}}|\tfk(y^{-1})| \,d\lambda^{w}(y) \not=0,
\end{equation*}
then as before there is a $z\in K$ such that $r(z)=w$ and
\begin{align*}
  \int_{K_{G}}|\tfk(y^{-1})|\,d\lambda^{w}(y) &=
  \int_{G}\charfcn{K_{G}} (zy) |\tfk(y^{-1}z^{-1})| \,d \lambda^{u}(y)
  \\
&= \int_{G}\charfcn{K_{G}}(zy^{-1})| \tfk(yz^{-1})| \,d\lambda_{u}(y)
\\
\intertext{which, since $K_{G}^{-1}=K_{G}$, is}
&= \int_{G}\charfcn{K_{G}}(yz^{-1}) |\tfk(yz^{-1})| \,d\lambda_{u}(y)
\\
&\le \int_{G}|F\bigl([y,z]\bigr)| \,d\lambda_{u}(y) \\
&=\int_{G(u)^{G}}|F\bigl([z,y]^{-1}\bigr) \,d\beta^{z\cdot
  G(u)}\bigl([z,y]\bigr) \\
&\le \|F\|_{I}.
\end{align*}
Using Lemma~\ref{lem-bdd}, we can find a neighborhood $V$ of $K_{G}$
contained in $C_{F}$ such that both
\begin{equation*}
  \int_{V}|\tfk(x)|\,d\lambda^{w}(x) \quad\text{and}\quad
\int_{V} |\tfk(x^{-1})|\,d\lambda^{w}(x)
\end{equation*}
are bounded by $\|F\|_{I}+1$ for all $w\in\go$.  Since $K_{G}$ is
symmetric, we can assume that $V=V^{-1}$ as well.  We can now let
$f_{K}$ be any element of $C_{c}(G)$ such that $f_{K}=\tfk$ on
$K_{G}$, $\supp f_{K}\subset V$ and $f_{K}\le\tfk$ everywhere.
Then $\|f\|_{I}\le\|F\|+1$, $\supp f_{K}\subset C_{F}$ and
$f_{K}(yz^{-1})= F\bigl([z,y]\bigr)$ for all $(z,y)\in K\times K$.
This completes the proof of the lemma.
\end{proof}

\begin{example}[Holonomy Groupoid]
  \label{ex-holonomy}
  Let $(V,\mathcal{F})$ be a $C^{\infty}$ compact foliated manifold,
  and let $G$ be its holonomy groupoid
  \citelist{\cite{win:agag83}\cite{hae:a84}} 
  equipped with its usual locally compact topology as in
  \cite{con:pspm80}.  Naturally, we also assume that $G$ is Hausdorff
  so that our results apply.  The stability groups $G(x)$ for $x\in V$
  are the holonomy groups for the foliation.  Using
  Theorem~\ref{thm-main}, each irreducible
  representation $\sigma_{x}$ of $G(x)$ provides an irreducible
  representation $\Ind_{G(x)}^{G}\sigma_{x}$ of $\cs(G)$.  This
  representation is equivalent to the representations
  $\Ind_{x}\sigma_{x}$ treated in \cite{facska:jot82}*{\S5}.
  Thus we recover a part of \cite{facska:jot82}*{Corollaire~5.7}.
\end{example}

\begin{example}
  \label{ex-deaconu}
  If $\sigma:X\to X$ is a covering map for a compact Hausdorff space
  $X$, then the associated \emph{Deaconu-Renualt groupoid}
  \cite{dea:tams95} is
  \begin{equation*}
    G:=\set{(z,n-l,w)\in X\times \Z\times
      X:\sigma^{l}(z)=\sigma^{n}(w)}. 
  \end{equation*}
More concretely, we can let $X$ be the circle $\T$ and
$\sigma(z):=z^{2}$.  Then the stability group $G(z)$ at $(z,0,z)$ is
trivial unless $z=1$ or $z$ is a primitive $2^{n}$-th root of unity.
Then $G(1)=\set{1,k,1):k\in\Z}$ and if $z$ is a primitive $2^{n}$-th
root of unity, then $G(z)=\set{(z,nk,z):k\in \Z}$.  Now applying
Theorem~\ref{thm-main}, we see that if $z$ is not a primitive root,
$\Ind_{G(z)}^{G}\delta(z,0,z)$ is an irreducible regular
representation.  If $z$ is a primitive $2^{n}$-th root of unity, then
for each $\omega\in \T\cong \widehat\Z$, we obtain an irreducible
representation $\Ind_{G(z)}^{G}\omega$. 
\end{example}

\section{Proof of Theorem~\ref{thm-stages}}

We let $\lambda$, $\beta$ and $\alpha$ be Haar systems on $G$, $K$ and
$H$, respectively. 
It will be helpful to notice that the space $\Hind$ of $\indhg L$ is
an internal tensor product $\X_{H}^{G}\tensor_{H}\H_{L}$ for the
appropriate actions of $\cs(H)$.\footnote{Internal tensor products of
  Hilbert modules are discussed in
  \citelist{\cite{lan:hilbert}\cite{wil:crossed}*{App.~I}}.} Thus the
space of $\Ind_{K}^{G}\bigl(\indgh L\bigr)$ is
$\X_{K}^{G}\tensor_{K}(\X_{H}^{K}\tensor_{H}\H_{L})$.  Of course, the
natural map on the algebraic tensor products induces an isomorphism
$U$ of $\X_{K}^{G}\tensor_{K}(\X_{H}^{K}\tensor_{H}\H_{L})$ with
$(\X_{K}^{G}\tensor_{K}\X_{H}^{K)}\tensor_{H}\H_{L}$ (see
\cite{wil:crossed}*{Lemma~I.6}).  
We need to combine this with the following observation.
\begin{lemma}
  \label{lem-rieffel-tech}
  The map sending $\phi\atensor\psi\in C_{c}(G_{K^{(0)}})\atensor
  C_{c}(K_{\ho})$ to $\theta(\phi\tensor\psi)$ in $C_{c}(G_{\ho})$,
  given by
  \begin{equation*}
    \theta(\phi\tensor\psi)(x):=
\int_{K}\phi(xk)\psi(k^{-1})\,d\beta^{s(x)}(k) ,
  \end{equation*}
induces an isomorphism, also called $\theta$, of $\X_{K}^{G}\tensor
\X_{H}^{K}$ onto $\X_{H}^{G}$.
\end{lemma}
\begin{proof}
  The first step is to see that $\theta$ is isometric.  Notice that we
  have three sets of actions and inner products.  We have not tried to
  invent notation to distinguish one from another.  Instead, we will
  hope that it is ``clear from context'' which formula is being
  employed.  In this spirit,
  \begin{align*}
    \Rip<\phi_{1}&\tensor \psi_{1},\phi_{2}\tensor \psi_{2}>(h) =
    \Rip<\psi_{1}, {\Rip<\phi_{1},\phi_{2}>\cdot \psi_{2}}>(h)\\
\intertext{which, using \eqref{eq:3} on $C_{c}(K_{\ho})$, is}
&=\int_{K}\overline{\psi_{1}(k)} \Rip<\phi_{1},\phi_{2}>\cdot \psi_{2}(kh) \,
d\beta_{r(h)}(k) \\
\intertext{which, using \eqref{eq:10} for the $C_{c}(K)$-action on
  $C_{c}(K_{\ho})$, is}
&=
\int_{K}\int_{K} \overline{\psi_{1}(k)}
\Rip<\phi_{1},\phi_{2}>(kk_{1})\psi_{2}
(k_{1}^{-1}k)\,d\beta^{r(h)}(k_{1}) \, d\beta_{r(h) }(k) \\
\intertext{which, using \eqref{eq:3}, is}
&=
\int_{K}\int_{K} \int_{G}\overline{\psi_{1}(k)\phi_{1}(x)} \phi_{2}(xkk_{1})
\psi_{2}
(k_{1}^{-1}k)\, d\lambda_{r(k)}(x) 
\,d\beta^{r(h)}(k_{1}) \, d\beta_{r(h) }(k)\\
&=\int_{K}\int_{K}\int_{G} \overline{\psi_{1}(k)\phi_{1}(xk^{-1})} \phi_{2}(xk_{1})
\psi_{2}
(k_{1}^{-1}k)\, d\lambda_{r(h)}(x) 
\,d\beta^{r(h)}(k_{1}) \, d\beta_{r(h) }(k)\\
\intertext{which, after using Fubini and sending $k_{1}$ to $hk_{1}$,
  is}
& =\int_{G}\overline{\theta(\phi_{1}\tensor\psi_{1})(x)}
\theta(\phi_{2}\tensor \psi_{2})(xh)\,d\lambda_{r(h)}(x) \\
&= \Rip<\theta(\phi_{1}\tensor\psi_{1}),\theta(\phi_{2}\tensor \psi_{2})>.
 \end{align*}
Thus, $\theta$ is isometric.  We just need to see that it has dense
range. 

However, notice that $\theta(\phi\tensor\psi)=\phi\cdot g_{\psi}$ for
the right action on $C_{c}(K)$ on $C_{c}(G_{K^{(0)}})$ with $g_{\psi}$
\emph{any} extension of $\psi$ to $C_{c}(K)$ (see \eqref{eq:2} and
Remark~\ref{rem-convolution}).  It follows from
\cite{mrw:jot87}*{Proposition~2.10} that there is an approximate
identity for $C_{c}(K)$ such that $\phi\cdot g_{i}\to \phi$ in the
inductive limit topology for all $\phi\in C_{c}(_{K^{(0)}})$.  This
implies that the range of $\theta$ is dense, and completes the proof
of the lemma.
\end{proof}

\begin{proof}[Proof of Theorem~\ref{thm-stages}]
  Define a unitary
  $V:\X_{K}^{G}\tensor_{K}(\X_{H}^{K}\tensor_{H}\H_{L})\to
  \X_{H}^{G}\tensor_{H}\H_{L}$ by $V=\theta\circ U$ (where $U$ is
  defined prior to Lemma~\ref{lem-rieffel-tech}).  Then on elementary
  tensors,
  $V\bigl(\phi\tensor(\psi\tensor
    h)\bigr)=\theta(\phi\tensor\psi)\tensor h$.
Then on the one hand,
\begin{equation*}
  V\bigl(\Ind_{K}^{G}\bigl(\Ind_{H}^{K}(L)(f)\bigr)\bigr)
  \bigl(\phi\tensor(\psi\tensor h)\bigr) =
  \theta(f*\phi\tensor\psi)\tensor h.
\end{equation*}
On the other hand, $\theta(f*\phi\tensor\psi)=f*\theta(\phi\tensor
\psi)$.  Therefore
\begin{equation*}
  V(\Ind_{K}^{G}\bigl(\Ind_{H}^{K}(L))=(\indgh L)V.
\end{equation*}
This completes the proof.
\end{proof}


\def\noopsort#1{}\def\cprime{$'$} \def\sp{^}

\end{document}